\documentclass[a4paper]{amsart}
\usepackage[latin1]{inputenc}
\usepackage{amssymb,amsmath,amsthm}
\usepackage{amsfonts}
\usepackage{paralist}
\usepackage{ifthen}
\usepackage{url}
\usepackage[english]{babel}

\theoremstyle{plain}
\newtheorem{theorem}{Theorem}

\newtheorem{lemma}[theorem]{Lemma}
\newtheorem{corollary}[theorem]{Corollary}

\theoremstyle{remark}

\newtheorem{remark}{Remark}

\theoremstyle{definition}
\newtheorem{example}{Example}

\newcommand{\card}[1]{\ensuremath{\lvert{#1}\rvert}}
\newcommand{\vect}[1]{\ensuremath{\mathbf{#1}}}

\begin{document}
\hyphenation{Bool-ean}

\title{Characterization of preclones by matrix collections}
\date{\today}
\author{Erkko Lehtonen}
\address{Université du Luxembourg \\
Faculté des Sciences, de la Technologie et de la Communication \\
6, rue Richard Coudenhove-Kalergi \\
L--1359 Luxembourg \\
Luxembourg}
\email{erkko.lehtonen@uni.lu}

\begin{abstract}
Preclones are described as the closed classes of the Galois connection induced by a preservation relation between operations and matrix collections. The Galois closed classes of matrix collections are also described by explicit closure conditions.
\end{abstract}

\maketitle


\section{Introduction}

The notion of preclone was introduced by Ésik and Weil~\cite{EW} in a study of recognizable tree languages. Preclones are heterogeneous algebras that resemble clones, but the superposition operation is slightly different from clone composition and membership of certain elements that are present in every clone is not stipulated. Precise definitions will be given in Section~\ref{sec:preliminaries}.

Clones have been described as the closed classes of operation under the Galois connection between operations and relations induced by the preservation relation. This classical Galois theory is known as the $\mathrm{Pol}$--$\mathrm{Inv}$ theory of clones and relations; see~\cite{BKKR,Geiger,Poschel,Szabo}. Similar Galois theories have been developed for other function algebras; see~\cite{CF,Harnau,Hellerstein,Lclusters,Pippenger}. We refer the reader to~\cite{Lclusters} for a brief survey on previous results in this line of research. For general background on clones and other function algebras, see~\cite{DW,Lau,PK,Szendrei}.

The purpose of the current paper is to characterize the preclones of operations in terms of a preservation relation between operations and certain dual objects. These dual objects are called matrix collections. The preservation relation induces a Galois connection between operations and matrix collections, and its closed classes of operations are precisely the locally closed preclones. We also present explicit closure conditions for matrix collections.


\section{Preliminaries}
\label{sec:preliminaries}

\subsection{Preclones}

A \emph{preclone} is a heterogeneous algebra
\[
\mathfrak{C} := \bigl((C^{(n)})_{n \geq 1}; (\ast^n_{m_1, \dots, m_n})_{n \geq 1, m_1, \dots, m_n \geq 1},\, \mathbf{1} \bigr)
\]
consisting of
\begin{enumerate}
\item infinitely many base sets, i.e., disjoint sets $C^{(n)}$ for $n \geq 1$,
\item operations $\ast^n_{m_1, \dots, m_n}$, called \emph{superpositions}, for all $n \geq 1$, $m_1, \dots, m_n \geq 1$, where $\ast^n_{m_1, \dots, m_n}$ is a map from $C^{(n)} \times C^{(m_1)} \times \dots \times C^{(m_n)}$ to $C^{(m)}$, where $m = \sum_{i = 1}^n m_i$
(in order to simplify notation, we will write $f \ast (g_1, \dots, g_n)$ for $\ast^n_{m_1, \dots, m_n}(f, g_1, \dots, g_n)$),
\item a distinguished element $\mathbf{1} \in C^{(1)}$;
\end{enumerate}
satisfying the following three equational axioms:
\begin{equation}
\bigl( f \ast (g_1, \dots, g_n) \bigr) \ast (h_1, \dots, h_m) =
f \ast (g_1 \ast \bar{h}_1, \dots, g_n \ast \bar{h}_n),
\tag{P1}\label{eq:assoc}
\end{equation}
where $f \in C^{(n)}$, $g_i \in C^{(m_i)}$ ($1 \leq i \leq n$), $m = \sum_{i = 1}^n m_i$, $h_j \in C^{\ell_j}$ ($1 \leq j \leq m$), and if we denote $\sum_{j = 1}^i m_i$ by $\hat{m}_i$ ($0 \leq i \leq n$), then $\bar{h}_i = (h_{\hat{m}_{i-1}+1}, \dots, h_{\hat{m}_i})$ ($1 \leq i \leq n$);
\begin{gather}
\mathbf{1} \ast f = f,
\tag{P2}\label{eq:idleft} \\
f \ast (\mathbf{1}, \dots, \mathbf{1}) = f,
\tag{P3}\label{eq:idright}
\end{gather}
where $f \in C^{(n)}$ and $\mathbf{1}$ appears $n$ times on the left-hand side of Axiom \eqref{eq:idright}.

Axiom \eqref{eq:assoc} is a generalization of associativity, and Axioms \eqref{eq:idleft} and \eqref{eq:idright} state that $\mathbf{1}$ is a neutral element.

An \emph{operation} on a nonempty set $A$ is a map $f \colon A^n \to A$ for some integer $n \geq 1$, called the \emph{arity} of $f$. We denote the set of all $n$-ary operations on $A$ by $\mathcal{O}_A^{(n)}$, and we let $\mathcal{O}_A := \bigcup_{n \geq 1} \mathcal{O}_A^{(n)}$. The $i$-th $n$-ary \emph{projection} is the operation $(a_1, a_2, \dotsc, a_n) \mapsto a_i$, and it is denoted by $x_i^{(n)}$.

It is an easy exercise to verify that we can obtain a preclone structure on $(\mathcal{O}_A^{(n)})_{n \geq 1}$ by defining the superposition operations $\ast^n_{m_1, \dots, m_n}$ as follows. For $f \in \mathcal{O}_A^{(n)}$, $g_i \in \mathcal{O}_A^{(m_i)}$ ($1 \leq i \leq n$), we let $\ast^n_{m_1, \dots, m_n}(f, g_1, \dots, g_n) := f \ast (g_1, \dots, g_n)$, where the operation $f \ast (g_1, \dots, g_n) \in \mathcal{O}_A^{(m)}$, $m = \sum_{i = 1}^n m_i$, is given by the rule
\begin{multline*}
\bigl( f \ast (g_1, \dotsc, g_n) \bigr) (a_{1,1}, \dotsc, a_{1,m_1}, a_{2,1}, \dotsc, a_{2,m_2}, \dotsc, a_{n,1}, \dotsc, a_{n,m_n}) = \\
f \bigl( g_1(a_{1,1}, \dotsc, a_{1,m_1}), g_2(a_{2,1}, \dotsc, a_{2,m_2}), \dotsc, g_n(a_{n,1}, \dotsc, a_{n,m_n}) \bigr),
\end{multline*}
for all $a_{i,j} \in A$, $1 \leq i \leq n$, $1 \leq j \leq m_i$. The first unary projection $x_1^{(1)}$ serves as the neutral element. The preclone $\bigl( (\mathcal{O}_A^{(n)})_{n \geq 1}; (\ast^n_{m_1, \dots, m_n})_{n \geq 1, m_1, \dots, m_n \geq 1},\, x_1^{(1)} \bigr)$ described above is called the \emph{full preclone of operations on $A$,} and its subalgebras are called \emph{preclones of operations on $A$.}

It is a well-known fact that every preclone is isomorphic to a preclone of operations on some set (see~\cite[Proposition 2.8]{EW} for a proof).

We conclude this subsection with a few examples of preclones. Further examples are provided in the paper by Ésik and Weil~\cite{EW}.

\begin{example}
\label{ex:clone}
The \emph{composition} of operation $f \in \mathcal{O}_A^{(n)}$ with $g_1, \dots, g_n \in \mathcal{O}_A^{(m)}$ is the operation $f \circ (g_1, \dots, g_n) \in \mathcal{O}_A^{(m)}$ given by the rule
\[
f \circ (g_1, \dots, g_n)(\mathbf{a}) := f \bigl( g_1(\mathbf{a}), \dots, g_n(\mathbf{a}) \bigr)
\quad
\text{for all $\mathbf{a} \in A^m$.}
\]
A \emph{clone} on $A$ is a set of operations on $A$ that is closed under composition and contains all the projections $x_i^{(n)}$ for all $n$ and $1 \leq i \leq n$. For general background on clones, see, e.g., \cite{DW,Lau,PK,Szendrei}.

Every clone on $A$ is (the universe of) a preclone of operations on $A$. For, let $\mathcal{C}$ be a clone on $A$. By definition, $\mathcal{C}$ contains the unary first projection. We want to verify that $\mathcal{C}$ is closed under superposition. Let $f \in \mathcal{C}^{(n)}$, $g_i \in \mathcal{C}^{(m_i)}$ ($1 \leq i \leq n$), $m := \sum_{i = 1}^n m_i$. Since $\mathcal{C}$ contains all projections and is closed under composition, $\mathcal{C}$ contains the $m$-ary operations
\[
g'_i := g_i \circ (x_{\hat{m}_{i-1} + 1}^{(m)}, x_{\hat{m}_{i-1} + 2}^{(m)}, \dots, x_{\hat{m}_i}^{(m)})
\qquad (1 \leq i \leq n),
\]
where $\hat{m}_i := \sum_{j = 1}^i m_i$ ($0 \leq i \leq n$), and we clearly have that
\[
f \ast (g_1, \dots, g_n) = f \circ (g'_1, \dots, g'_n),
\]
which is a member of $\mathcal{C}$ since $\mathcal{C}$ is closed under composition.
\end{example}

\begin{example}
Let $\mathcal{A} := (A; (f^\mathcal{A}_i)_{i \in I})$ be an algebra of type $\tau$. It is well-known that the set $W_\tau(X)^\mathcal{A}$ of term operations on $\mathcal{A}$ is a clone (see~\cite{DW,Lau,PK,Szendrei}) and hence it is a preclone by Example~\ref{ex:clone}. Consider the following subsets of $W_\tau(X)^\mathcal{A}$:
\begin{itemize}
\item the set $W^\mathrm{lin}_\tau(X)^\mathcal{A}$ of term operations on $\mathcal{A}$ induced by \emph{linear terms,} i.e., terms of type $\tau$ where no variable occurs more than once;
\item the set $W^\mathrm{inc}_\tau(X)^\mathcal{A}$ of term operations on $\mathcal{A}$ induced by linear terms where the variables occur in increasing order, i.e., if $t$ is a linear term of type $\tau$ and variables $x_i$ and $x_j$ occur in $t$ and the occurrence of $x_i$ in $t$ is to the left of that of $x_j$, then $i < j$.
\end{itemize}
It is easy to verify that $W^\mathrm{lin}_\tau(X)^\mathcal{A}$ and $W^\mathrm{inc}_\tau(X)^\mathcal{A}$ are preclones, but they are not in general clones.
\end{example}


\subsection{A Galois connection between operations and matrix collections}

A \emph{Galois connection} between sets $A$ and $B$ is a pair $(\sigma, \tau)$ of mappings $\sigma \colon \mathcal{P}(A) \to \mathcal{P}(B)$ and $\tau \colon \mathcal{P}(B) \to \mathcal{P}(A)$ between the power sets $\mathcal{P}(A)$ and $\mathcal{P}(B)$ such that for all $X, X' \subseteq A$ and all $Y, Y' \subseteq B$ the following conditions are satisfied:
\[
\begin{array}{c}
X \subseteq X' \Longrightarrow \sigma(X) \supseteq \sigma(X'), \\
Y \subseteq Y' \Longrightarrow \tau(Y) \supseteq \tau(Y'),
\end{array}
\quad\text{and}\quad
\begin{array}{c}
X \subseteq \tau(\sigma(X)), \\
Y \subseteq \sigma(\tau(Y)),
\end{array}
\]
or, equivalently,
\[
X \subseteq \tau(Y) \Longleftrightarrow \sigma(X) \supseteq Y.
\]

The most popular Galois connections are derived from binary relations, as the following well-known theorem shows (for early references, see~\cite{Everett,Ore}; see also~\cite{DEW,Lau}):
\begin{theorem}
\label{thm:Galois}
Let $A$ and $B$ be nonempty sets and let $R \subseteq A \times B$. Define the mappings $\sigma \colon \mathcal{P}(A) \to \mathcal{P}(B)$, $\tau \colon \mathcal{P}(B) \to \mathcal{P}(A)$ by
\begin{align*}
\sigma(X) &:= \{y \in B \mid \forall x \in X \colon (x, y) \in R\}, \\
\tau(Y) &:= \{x \in A \mid \forall y \in Y \colon (x, y) \in R\}.
\end{align*}
Then the pair $(\sigma, \tau)$ is a Galois connection between $A$ and $B$.
\end{theorem}

For any nonnegative integers $m, n$, we denote by $A^{m \times n}$ the set of all matrices with $m$ rows and $n$ columns and entries from $A$. Subsets $\Gamma \subseteq \bigcup_{p \geq 0} A^{m \times p}$, for a fixed $m \geq 1$, are called \emph{matrix collections} on $A$, and the number $m$ is referred to as the \emph{arity} of $\Gamma$. For $m \geq 1$, we denote
\[
\mathcal{M}_A^{(m)} := \{\Gamma \mid \Gamma \subseteq \bigcup_{p \geq 0} A^{m \times p}\}
\quad\text{and}\quad
\mathcal{M}_A := \bigcup_{m \geq 1} \mathcal{M}_A^{(m)}.
\]
The \emph{breadth} of a matrix collection $\Gamma$ is the maximum number of columns of the matrices that are members of $\Gamma$, provided that this maximum exists; if there is no maximum, we say that $\Gamma$ has \emph{infinite breadth.} We also agree that the breadth of the empty matrix collection $\emptyset$ is $0$.

Let $f \in \mathcal{O}_A^{(n)}$, and let $\mathbf{M} := (a_{ij}) \in A^{m \times n}$. We denote by $f \mathbf{M}$ the $m$-tuple
\[
(f(a_{11}, a_{12}, \dotsc, a_{1n}), f(a_{21}, a_{22}, \dotsc, a_{2n}), \dotsc, f(a_{m1}, a_{m2}, \dotsc, a_{mn})),
\]
i.e., the $m$-tuple obtained by applying $f$ to the rows of $\mathbf{M}$. We will interpret $f \mathbf{M}$ as a column vector.
We say that an operation $f \in \mathcal{O}_A^{(n)}$ \emph{preserves} a matrix collection $\Gamma \in \mathcal{M}_A^{(m)}$, denoted $f \vartriangleright \Gamma$, if for all $m$-row matrices $\mathbf{M} := [\mathbf{M}_1 | \mathbf{M}_2 | \mathbf{M}_3]$, where $\mathbf{M}_2$ has $n$ columns, the condition $\mathbf{M} \in \Gamma$ implies $[\mathbf{M}_1 | f \mathbf{M}_2 | \mathbf{M}_3] \in \Gamma$.

Let $\mathcal{M} \subseteq \mathcal{M}_A$ be a set of matrix collections on $A$, and let $\mathcal{F} \subseteq \mathcal{O}_A$ be a set of operations on $A$. We say that $\mathcal{F}$ is \emph{characterized} by $\mathcal{M}$, if $\mathcal{F} = \{f \in \mathcal{O}_A \mid \forall \Gamma \in \mathcal{M} \colon f \vartriangleright \Gamma\}$, i.e., $\mathcal{F}$ is precisely the set of all operations that preserve all matrix collections in $\mathcal{M}$. Similarly, we say that $\mathcal{M}$ is \emph{characterized} by $\mathcal{F}$, if $\mathcal{M} = \{\Gamma \in \mathcal{M}_A \mid \forall f \in \mathcal{F} \colon f \vartriangleright \Gamma\}$, i.e., $\mathcal{M}$ is precisely the set of all matrix collections that are preserved by all operations in $\mathcal{F}$. In light of Theorem~\ref{thm:Galois}, the relation $\vartriangleright$ induces a Galois connection between $\mathcal{O}_A$ and $\mathcal{M}_A$, and its closed classes of operations (matrix collections) are exactly those which are characterized by matrix collections (operations, respectively).

In the remaining two sections, we will show that the closed classes of operations are precisely the locally closed preclones (Theorem~\ref{thm:preclones}), and we will describe (in Theorem~\ref{thm:bundles}) the closed classes of matrix collections as subsets of $\mathcal{M}_A$ that are closed under certain operations on matrix collections that will be defined in Section~\ref{sec:mc}


\section{Preclones are characterized by matrix collections}

In this section, we will show that the sets of operations on $A$ that are characterized by matrix collections are precisely the universes of preclones of operations on $A$ that are locally closed.

We say that a set $\mathcal{F}$ of operations on $A$ is \emph{locally closed} if for all $f \colon A^n \to A$ it holds that $f \in \mathcal{F}$ whenever for all finite subsets $S$ of $A^n$, there exists a $g \in \mathcal{F}$ such that $f|_S = g|_S$. (Note that every set of operations on a finite set is locally closed by definition.)

\begin{theorem}
\label{thm:preclones}
Let $\mathcal{F} \subseteq \mathcal{O}_A$ be a set of operations on $A$. The following are equivalent.
\begin{enumerate}[\rm (i)]
\item $\mathcal{F}$ is the universe of a preclone of operations on $A$ that is locally closed.
\item $\mathcal{F}$ is characterized by some set of matrix collections on $A$.
\end{enumerate}
\end{theorem}
\begin{proof}
(ii) $\Rightarrow$ (i): Assume that $\mathcal{F}$ is characterized by a set $\mathcal{M} \subseteq \mathcal{M}_A$ of matrix collections. Let $\Gamma \in \mathcal{M}$. For all matrices $\mathbf{M} := [\mathbf{M}_1 | \mathbf{M}_2 | \mathbf{M}_3] \in \Gamma$ such that $\mathbf{M}_2$ has exactly one column, we have that $[\mathbf{M}_1 | x_1^{(1)} \mathbf{M}_2 | \mathbf{M}_3] = [\mathbf{M}_1 | \mathbf{M}_2 | \mathbf{M}_3] \in \Gamma$, and hence $x_1^{(1)} \vartriangleright \Gamma$. Therefore $x_1^{(1)}$ preserves every $\Gamma \in \mathcal{M}$, so $x_1^{(1)} \in \mathcal{F}$.

Let $f \in \mathcal{F}^{(n)}$, $g_i \in \mathcal{F}^{(m_i)}$ ($1 \leq i \leq n$), and let $m = \sum_{i=1}^n m_i$. We will show that $f \ast (g_1, \dotsc, g_n) \in \mathcal{F}$. Let $\Gamma \in \mathcal{M}$, and let $\mathbf{M} := [\mathbf{M}_1 | \mathbf{M}_2 | \mathbf{M}_3] \in \Gamma$ such that $\mathbf{M}_2$ has $m$ columns. Let $[\mathbf{M}_{2,1} | \mathbf{M}_{2,2} | \dotsb | \mathbf{M}_{2,n}] := \mathbf{M}_2$ such that $\mathbf{M}_{2,i}$ has $m_i$ columns ($1 \leq i \leq n$); thus $\mathbf{M} = [\mathbf{M}_1 | \mathbf{M}_{2,1} | \mathbf{M}_{2,2} | \dotsb | \mathbf{M}_{2,n} | \mathbf{M}_3]$. By our assumption that $g_i \vartriangleright \Gamma$ for $1 \leq i \leq n$, a simple inductive proof shows that $[\mathbf{M}_1 | g_1 \mathbf{M}_{2,1} | g_2 \mathbf{M}_{2,2} | \dotsb | g_n \mathbf{M}_{2,n} | \mathbf{M}_3] \in \Gamma$. Since $f \vartriangleright \Gamma$, we have that
\[
[\mathbf{M}_1 | f \ast (g_1, \dotsc, g_n) \mathbf{M}_2 | \mathbf{M}_3] =
[\mathbf{M}_1 | f [g_1 \mathbf{M}_{2,1} | \dotsb | g_n \mathbf{M}_{2,n}] | \mathbf{M}_3] \in \Gamma.
\]

It remains to show that $\mathcal{F}$ is locally closed. Suppose on the contrary that there is a $f \in \mathcal{O}_A \setminus \mathcal{F}$, say $n$-ary, such that for all finite subsets $F \subseteq A^n$ there exists a $g \in \mathcal{F}$ satisfying $g|_F = f|_F$. Since $f \notin \mathcal{F}$, there is a matrix collection $\Gamma \in \mathcal{M}$ and a matrix $[\mathbf{M}_1 | \mathbf{M}_2 | \mathbf{M}_3] \in \Gamma$ such that $[\mathbf{M}_1 | f \mathbf{M}_2 | \mathbf{M}_3] \notin \Gamma$. Let $F$ be the finite set of rows of $\mathbf{M}_2$. By our assumption, there exists a function $g \in \mathcal{F}$ such that $g|_F = f|_F$, and so we have that $g \mathbf{M}_2 = g|_F \mathbf{M}_2 = f|_F \mathbf{M}_2 = f \mathbf{M}_2$. Hence $[\mathbf{M}_1 | g \mathbf{M}_2 | \mathbf{M}_3] \notin \Gamma$, which contradicts the fact that $g \vartriangleright \Gamma$.

(ii) $\Rightarrow$ (i): Assume that $\mathcal{F}$ is a locally closed preclone. We will show that for each $g \notin \mathcal{F}$, there is a matrix collection $\Gamma$ such that $g \ntriangleright \Gamma$ but for every $f \in \mathcal{F}$, $f \vartriangleright \Gamma$. The set of all such ``separating'' matrix collections for every $g \notin \mathcal{F}$ characterizes $\mathcal{F}$.

Assume that $g \notin \mathcal{F}$ is $m$-ary. Since $\mathcal{F}$ is locally closed, there is a finite subset $S \subseteq A^m$ such that $g|_S \neq f|_S$ for every $m$-ary $f \in \mathcal{F}$. Clearly $S$ is nonempty. Let $\mathbf{M}_*$ be an $\card{S} \times m$ matrix whose rows are the elements of $S$ in some fixed order. Let
\[
\Gamma := \bigl\{[h_1 \mathbf{M}_1 | \dotsb | h_r \mathbf{M}_r] \bigm\vert \mathbf{M}_* = [\mathbf{M}_1 | \dotsb | \mathbf{M}_r],\, r \geq 1,\, h_1, \dotsc, h_r \in \mathcal{F} \bigr\},
\]
where the number of columns of each $\mathbf{M}_i$ equals the arity of $h_i$ ($1 \leq i \leq r$). (Note that $\Gamma$ contains the matrix $\mathbf{M}_*$, because $x_1^{(1)} \in \mathcal{F}$ by the assumption that $\mathcal{F}$ is the universe of a preclone.)

By the definition of $\mathbf{M}_*$, $g \mathbf{M}_* \neq f \mathbf{M}_*$ for every ($m$-ary) $f \in \mathcal{F}$, and hence $g \mathbf{M}_* \notin \Gamma$; thus $g \ntriangleright \Gamma$. We still need to show that $f \vartriangleright \Gamma$ for all $f \in \mathcal{F}$. Let $f \in \mathcal{F}$ be $n$-ary and $\mathbf{M} \in \Gamma$. Then there exist $r \geq 1$ and $h_1, \dotsc, h_r \in \mathcal{F}$ such that $\mathbf{M} = [h_1 \mathbf{M}_1 | \dotsb | h_r \mathbf{M}_r]$ where $[\mathbf{M}_1 | \dotsb | \mathbf{M}_r] = \mathbf{M}_*$. Let $[\mathbf{M}'_1 | \mathbf{M}'_2 | \mathbf{M}'_3] := \mathbf{M}$ where $\mathbf{M}'_2$ has $n$ columns. Then
\begin{align*}
&\!\!\![\mathbf{M}'_1 | f \mathbf{M}'_2 | \mathbf{M}'_3] \\
&= [h_1 \mathbf{M}_1 | \dotsb | h_p \mathbf{M}_p | f [h_{p+1} \mathbf{M}_{p+1} | \dotsb | h_{p+n} \mathbf{M}_{p+n}] | h_{p+n+1} \mathbf{M}_{p+n+1} | \dotsb | h_r \mathbf{M}_r] \\
&= [h_1 \mathbf{M}_1 | \dotsb | h_p \mathbf{M}_p | \\
& \qquad\qquad f \ast (h_{p+1}, \dotsc, h_{p+n}) [\mathbf{M}_{p+1} | \dotsb | \mathbf{M}_{p+n}] | h_{p+n+1} \mathbf{M}_{p+n+1} | \dotsb | h_r \mathbf{M}_r]
\end{align*}
for some $p \geq 0$. We have that $f \ast (h_{p+1}, \dotsc, h_{p+n}) \in \mathcal{F}$ by the assumption that $\mathcal{F}$ is the universe of a preclone, and hence the matrix in the last line of the displayed chain of equalities is in $\Gamma$ by the definition of $\Gamma$.
\end{proof}


\section{Closure conditions for matrix collections}
\label{sec:mc}

In this section, we will establish explicit closure conditions for sets of matrix collections that are characterized by sets of operations. We will introduce a number of operations on the set $\mathcal{M}_A$ of matrix collections on $A$, and we will show that the closed subsets of $\mathcal{M}_A$ are precisely the subsets that are closed under these operations. Our methods and proofs follow closely those employed in~\cite{Lclusters}, which in turn are adaptations of those by Couceiro and Foldes~\cite{CF}.

For maps $f \colon A \to B$ and $g \colon C \to D$, the composition $g \circ f$ is defined only if the codomain $B$ of $f$ coincides with the domain $C$ of $g$. Removing this restriction, the \emph{concatenation} of $f$ and $g$ is defined to be the map $gf \colon f^{-1}[B \cap C] \to D$ given by the rule $(gf)(a) := g(f(a))$ for all $a \in f^{-1}[B \cap C]$. Clearly, if $B = C$, then $gf = g \circ f$; thus functional composition is subsumed and extended by concatenation. Concatenation is associative, i.e., for any maps $f$, $g$, $h$, we have $h(gf) = (hg)f$.

For a family $(g_i)_{i \in I}$ of maps $g_i \colon A_i \to B_i$ such that $A_i \cap A_j = \emptyset$ whenever $i \neq j$, we define the (\emph{piecewise}) \emph{sum of the family $(g_i)_{i \in I}$} to be the map $\sum_{i \in I} g_i \colon \bigcup_{i \in I} A_i \to \bigcup_{i \in I} B_i$ whose restriction to each $A_i$ coincides with $g_i$. If $I$ is a two-element set, say $I := \{1, 2\}$, then we write $g_1 + g_2$. Clearly, this operation is associative and commutative.

Concatenation is distributive over summation, i.e., for any family $(g_i)_{i \in I}$ of maps on disjoint domains and any map $f$,
\[
\Bigl( \sum_{i \in I} g_i \Bigr) f = \sum_{i \in I} (g_i f)
\qquad \text{and} \qquad
f \Bigl( \sum_{i \in I} g_i \Bigr) = \sum_{i \in I} (f g_i).
\]
In particular, if $g_1$ and $g_2$ are maps with disjoint domains, then
\[
(g_1 + g_2) f = (g_1 f) + (g_2 f)
\qquad \text{and} \qquad
f(g_1 + g_2) = (f g_1) + (f g_2).
\]

Let $m$ and $n$ be positive integers (viewed as ordinals, i.e., $m := \{0, \dotsc, m - 1\}$). Let $h \colon n \to m \cup V$ where $V$ is an arbitrary set of symbols disjoint from the ordinals, called \emph{existentially quantified indeterminate indices,} or simply \emph{indeterminates,} and let $\sigma \colon V \to A$ be any map, called a \emph{Skolem map.} Then each $m$-tuple $\vect{a} \in A^m$, being a map $\vect{a} \colon m \to A$, gives rise to an $n$-tuple $(\vect{a} + \sigma) h =: (b_0, \dots, b_{n-1}) \in A^n$, where
\[
b_i :=
\begin{cases}
a_{h(i)}, & \text{if $h(i) \in \{0, 1, \dots, m - 1\}$,} \\
\sigma(h(i)), & \text{if $h(i) \in V$.}
\end{cases}
\]

Let $H = (h_j)_{j \in J}$ be a nonempty family of maps $h_j \colon n_j \to m \cup V$, where each $n_j$ is a positive integer. Then $H$ is called a \emph{minor formation scheme} with \emph{target} $m$, \emph{indeterminate set} $V$, and \emph{source family} $(n_j)_{j \in J}$. Let $(\Gamma_j)_{j \in J}$ be a family of matrix collections on $A$, each $\Gamma_j$ of arity $n_j$, and let $\Gamma$ be an $m$-ary matrix collection on $A$. We say that $\Gamma$ is a \emph{conjunctive minor} of the family $(\Gamma_j)_{j \in J}$ \emph{via} $H$, if, for every $m \times n$ matrix $\mathbf{M} := (\vect{a}_{*1}, \dotsc, \vect{a}_{*n}) \in A^{m \times n}$,
\[
\mathbf{M} \in \Gamma
\Longleftrightarrow
\bigl[
\exists \sigma_1, \dots, \sigma_n \in A^V \:
\forall j \in J \colon
\bigl( (\vect{a}_{*1} + \sigma_1) h_j, \dotsc, (\vect{a}_{*n} + \sigma_n) h_j \bigr) \in \Gamma_j
\bigr].
\]
In the case that the minor formation scheme $H := (h_j)_{j \in J}$ and the family $(\Gamma_j)_{j \in J}$ are indexed by a singleton $J := \{0\}$, a conjunctive minor $\Gamma$ of a family consisting of a single matrix collection $\Gamma_0$ is called a \emph{simple minor} of $\Gamma_0$.

The formation of conjunctive minors subsumes the formation of simple minors and the intersection of collections of matrices. Simple minors in turn subsume permutation of rows, projection, identification of rows, and addition of inessential rows, operations which can be defined for matrix collections in an analogous way as for generalized constraints (cf.~\cite{Hellerstein,Lclusters}).

\begin{lemma}
\label{lemma:conjmin}
Let $\Gamma$ be a conjunctive minor of a nonempty family $(\Gamma_j)_{j \in J}$ of matrix collections on $A$. If $f \colon A^n \to A$ preserves $\Gamma_j$ for all $j \in J$, then $f$ preserves $\Gamma$.
\end{lemma}
\begin{proof}
Let $\Gamma$ be an $m$-ary conjunctive minor of the family $(\Gamma_j)_{j \in J}$ via the scheme $H := (h_j)_{j \in J}$, $h_j \colon n_j \to m \cup V$. Let $\mathbf{M} := (\vect{a}_{*1}, \dotsc, \vect{a}_{*p})$ be an $m \times p$ matrix $(p \geq n)$ such that $\mathbf{M} \in \Gamma$ and denote $\mathbf{M}_1 := (\vect{a}_{*1}, \dotsc, \vect{a}_{*q})$, $\mathbf{M}_2 := (\vect{a}_{*(q+1)}, \dotsc, \vect{a}_{*(q+n)})$, $\mathbf{M}_3 := (\vect{a}_{*(q+n+1)}, \dotsc, \vect{a}_{*p})$, for some $0 \leq q \leq p-n$, so $\mathbf{M} = [\mathbf{M}_1 | \mathbf{M}_2 | \mathbf{M}_3]$. We need to prove that $[\mathbf{M}_1 | f \mathbf{M}_2 | \mathbf{M}_3] \in \Gamma$.

Since $\Gamma$ is a conjunctive minor of $(\Gamma_j)_{j \in J}$ via $H = (h_j)_{j \in J}$, there are Skolem maps $\sigma_i \colon V \to A$ ($1 \leq i \leq p$) such that for every $j \in J$, we have
\[
\bigl( (\vect{a}_{*1} + \sigma_1) h_j, \dotsc, (\vect{a}_{*p} + \sigma_p) h_j \bigr) \in \Gamma_j.
\]
For each $j \in J$, denote
\begin{align*}
\mathbf{M}_1^j &:= \bigl( (\vect{a}_{*1} + \sigma_1) h_j, \dotsc, (\vect{a}_{*q} + \sigma_q) h_j \bigr), \\
\mathbf{M}_2^j &:= \bigl( (\vect{a}_{*(q+1)} + \sigma_{q+1}) h_j, \dotsc, (\vect{a}_{*(q+n)} + \sigma_{q+n}) h_j \bigr), \\
\mathbf{M}_3^j &:= \bigl( (\vect{a}_{*(q+n+1)} + \sigma_{q+n+1}) h_j, \dotsc, (\vect{a}_{*p} + \sigma_p) h_j \bigr).
\end{align*}
Let $\sigma := f(\sigma_{q+1}, \dots, \sigma_{q+n})$. By the distributivity of concatenation over piecewise sum of mappings and by the associativity of concatenation, we have that, for each $j \in J$,
\[
\begin{split}
(f \mathbf{M}_2 + \sigma) h_j
&= \bigl( (f(\vect{a}_{*(q+1)}, \dotsc, \vect{a}_{*(q+n)}) + f(\sigma_{q+1}, \dotsc, \sigma_{q+n}) \bigr) h_j \\
&= \bigl( f(\vect{a}_{*(q+1)} + \sigma_{q+1}, \dotsc, \vect{a}_{*(q+n)} + \sigma_{q+n}) \bigr) h_j \\
&= f \bigl( (\vect{a}_{*(q+1)} + \sigma_{q+1}) h_j, \dotsc, (\vect{a}_{*(q+n)} + \sigma_{q+n}) h_j \bigr)
= f \mathbf{M}_2^j.
\end{split}
\]
Since $f$ is assumed to preserve $\Gamma_j$, we have that $[\mathbf{M}_1^j | f \mathbf{M}_2^j | \mathbf{M}_3^j] \in \Gamma_j$ for each $j \in J$. Since $\Gamma$ is a conjunctive minor of $(\Gamma_j)_{j \in J}$ via $H = (h_j)_{j \in J}$, this implies that $[\mathbf{M}_1 | f \mathbf{M}_2 | \mathbf{M}_3] \in \Gamma$. Thus, $f \vartriangleright \Gamma$.
\end{proof}

\begin{lemma}
\label{lemma:union}
Let $(\Gamma_j)_{j \in J}$ be a nonempty family of $m$-ary matrix collections on $A$. If $f \colon A^n \to A$ preserves $\Gamma_j$ for all $j \in J$, then $f$ preserves $\bigcup_{j \in J} \Gamma_j$.
\end{lemma}
\begin{proof}
Let $\mathbf{M} = [\mathbf{M}_1 | \mathbf{M}_2 | \mathbf{M}_3] \in \bigcup_{j \in J} \Gamma_j$ be such that $\mathbf{M}_2$ has $n$ columns. Then there is an $i \in J$ such that $\mathbf{M} \in \Gamma_i$. By the assumption that $f \vartriangleright \Gamma_i$, we have that $[\mathbf{M}_1 | f \mathbf{M}_2 | \mathbf{M}_3] \in \Gamma_i$, and hence $[\mathbf{M}_1 | f \mathbf{M}_2 | \mathbf{M}_3] \in \bigcup_{j \in J} \Gamma_j$.
\end{proof}

The \emph{right quotient} of an $m$-ary matrix collection $\Gamma$ on $A$ by an $m \times n$ matrix $\mathbf{N}$ is defined by
\[
\Gamma / \mathbf{N} := \{\mathbf{M} \mid [\mathbf{M} | \mathbf{N}] \in \Gamma\}.
\]
The \emph{left quotient} of $\Gamma$ by $\mathbf{N}$ is defined similarly:
\[
\mathbf{N} \backslash \Gamma := \{\mathbf{M} \mid [\mathbf{N} | \mathbf{M}] \in \Gamma\}.
\]

\begin{lemma}
\label{lemma:quotprop}
Let $\Gamma, \Gamma' \in \mathcal{M}_A^{(m)}$ and $\mathbf{N} \in A^{m \times n}$, $\mathbf{N}' \in A^{m \times n'}$.
\begin{compactenum}[(i)]
\item $\mathbf{M} \in \Gamma / \mathbf{N}$ if and only if $[\mathbf{M} | \mathbf{N}] \in \Gamma$.
\item $\mathbf{M} \in \mathbf{N} \backslash \Gamma$ if and only if $[\mathbf{N} | \mathbf{M}] \in \Gamma$.
\item $(\mathbf{N} \backslash \Gamma) / \mathbf{N}' = \mathbf{N} \backslash (\Gamma / \mathbf{N}')$.
\item $\mathbf{N} \backslash (\Gamma \cup \Gamma') / \mathbf{N}' = (\mathbf{N} \backslash \Gamma / \mathbf{N}') \cup (\mathbf{N} \backslash \Gamma' / \mathbf{N}')$.
\end{compactenum}
\end{lemma}
\begin{proof}
(i), (ii): Immediate from the definition of right and left quotients.

(iii): By parts (i) and (ii), we have
\begin{multline*}
\mathbf{M} \in (\mathbf{N} \backslash \Gamma) / \mathbf{N}'
\Longleftrightarrow
[\mathbf{M} | \mathbf{N}'] \in \mathbf{N} \backslash \Gamma
\Longleftrightarrow
[\mathbf{N} | \mathbf{M} | \mathbf{N}'] \in \Gamma \\
\Longleftrightarrow
[\mathbf{N} | \mathbf{M}] \in \Gamma / \mathbf{N}'
\Longleftrightarrow
[\mathbf{N} | \mathbf{M} | \mathbf{N}'] \in \mathbf{N} \backslash (\Gamma / \mathbf{N}'),
\end{multline*}
and the claim follows.

(iv): By parts (i) and (ii) and by the definition of union, we have
\begin{align*}
\mathbf{M} \in \mathbf{N} \backslash (\Gamma \cup \Gamma') / \mathbf{N}'
&\Longleftrightarrow
[\mathbf{N} | \mathbf{M} | \mathbf{N}'] \in \Gamma \cup \Gamma' \\
&\Longleftrightarrow
[\mathbf{N} | \mathbf{M} | \mathbf{N}'] \in \Gamma \vee [\mathbf{N} | \mathbf{M} | \mathbf{N}'] \in \Gamma' \\
&\Longleftrightarrow
\mathbf{M} \in \mathbf{N} \backslash \Gamma / \mathbf{N}' \vee \mathbf{M} \in \mathbf{N} \backslash \Gamma' / \mathbf{N}' \\
&\Longleftrightarrow
\mathbf{M} \in (\mathbf{N} \backslash \Gamma / \mathbf{N}') \cup (\mathbf{N} \backslash \Gamma' / \mathbf{N}'),
\end{align*}
and the claim follows.
\end{proof}

\begin{remark}
By Lemma~\ref{lemma:quotprop}, $(\mathbf{N} \backslash \Gamma) / \mathbf{N}' = \mathbf{N} \backslash (\Gamma / \mathbf{N}')$ and hence we can write $\mathbf{N} \backslash \Gamma / \mathbf{N}'$ without ambiguity. By parts (i) and (ii), we also have that $\mathbf{M} \in \mathbf{N} \backslash \Gamma / \mathbf{N}'$ if and only if $[\mathbf{N} | \mathbf{M} | \mathbf{N}'] \in \Gamma$.
\end{remark}

\begin{lemma}
\label{lemma:quotient}
Let $\Gamma$ be an $m$-ary matrix collection on $A$. If $f$ preserves $\Gamma$, then $f$ preserves $\mathbf{N} \backslash \Gamma$ and $\Gamma / \mathbf{N}$ for all $m$-row matrices $\mathbf{N}$.
\end{lemma}
\begin{proof}
Assume that $f \in \mathcal{O}_A^{(n)}$ preserves $\Gamma$. Let $[\mathbf{M}_1 | \mathbf{M}_2 | \mathbf{M}_3] \in \mathbf{N} \backslash \Gamma$ such that $\mathbf{M}_2$ has $n$ columns. By Lemma~\ref{lemma:quotprop}, $[\mathbf{N} | \mathbf{M}_1 | \mathbf{M}_2 | \mathbf{M}_3] \in \Gamma$, and by the assumption that $f \vartriangleright \Gamma$ we have that $[\mathbf{N} | \mathbf{M}_1 | f \mathbf{M}_2 | \mathbf{M}_3] \in \Gamma$. Again, by Lemma~\ref{lemma:quotprop}, $[\mathbf{M}_1 | f \mathbf{M}_2 | \mathbf{M}_3] \in \mathbf{N} \backslash \Gamma$, and we conclude that $f \vartriangleright \mathbf{N} \backslash \Gamma$. The statement $f \vartriangleright \Gamma / \mathbf{N}$ is proved similarly.
\end{proof}

\begin{lemma}
\label{lemma:dividend}
Assume that $\Gamma \in \mathcal{M}_A^{(m)}$ contains all $m$-row matrices on $A$ with at most $p$ columns, for some $p \geq 0$. If $f$ preserves $\mathbf{N}_1 \backslash \Gamma / \mathbf{N}_2$ for all matrices $\mathbf{N}_1$, $\mathbf{N}_2$ such that $[\mathbf{N}_1 | \mathbf{N}_2]$ has at least $p$ columns, then $f$ preserves $\Gamma$.
\end{lemma}
\begin{proof}
Assume that $f \in \mathcal{O}_A^{(n)}$ satisfies the hypotheses of the lemma. Let $\mathbf{M} := [\mathbf{M}_1 | \mathbf{M}_2 | \mathbf{M}_3] \in \Gamma$ where $\mathbf{M}_i$ has $n_i$ columns ($i = 1, 2, 3$) and $n_2 = n$. If $n_1 + n_3 < p$, then $[\mathbf{M}_1 | f \mathbf{M}_2 | \mathbf{M}_3]$ has $n_1 + n_3 + 1 \leq p$ columns and is obviously a member of $\Gamma$. We can thus assume that $n_1 + n_3 \geq 3$. By Lemma~\ref{lemma:quotprop}, $\mathbf{M}_2 \in \mathbf{M}_1 \backslash \Gamma / \mathbf{M}_3$, and so $f \mathbf{M}_2 \in \mathbf{M}_1 \backslash \Gamma / \mathbf{M}_3$ by our assumptions. Using Lemma~\ref{lemma:quotprop} again, we conclude that $[\mathbf{M}_1 | f \mathbf{M}_2 | \mathbf{M}_3] \in \Gamma$.
\end{proof}

For $p \geq 0$, the $m$-ary \emph{trivial matrix collection of breadth $p$,} denoted $\Omega_m^{(p)}$, is the set of all $m$-row matrices on $A$ with at most $p$ columns. The \emph{empty matrix collection} (of any arity) is the empty set $\emptyset$. Note that $\Omega_m^{(0)} \neq \emptyset$, because the empty matrix is the unique member of $\Omega_m^{(0)}$. The binary \emph{equality matrix collection,} denoted $E_2$, is the set of all two-row matrices with any finite number of columns such that the two rows are identical.

For $p \geq 0$, we say that the matrix collection $\Gamma^{(p)} := \Gamma \cap \Omega_m^{(p)}$ is obtained from the $m$-ary matrix collection $\Gamma$ by \emph{restricting the breadth to $p$.}

\begin{lemma}
\label{lemma:breadth}
Let $\Gamma$ be an $m$-ary matrix collection on $A$. Then $f$ preserves $\Gamma$ if and only if $f$ preserves $\Gamma^{(p)}$ for all $p \geq 0$.
\end{lemma}
\begin{proof}
Assume first that $f \vartriangleright \Gamma$. Let $[\mathbf{M}_1 | \mathbf{M}_2 | \mathbf{M}_3] \in \Gamma^{(p)}$ for some $p \geq 0$. Since $\Gamma^{(p)} \subseteq \Gamma$, we have that $[\mathbf{M}_1 | \mathbf{M}_2 | \mathbf{M}_3] \in \Gamma$ and hence $[\mathbf{M}_1 | f \mathbf{M}_2 | \mathbf{M}_3] \in \Gamma$ by our assumption. The number of columns in $[\mathbf{M}_1 | f \mathbf{M}_2 | \mathbf{M}_3]$ is at most $p$, so we have that $[\mathbf{M}_1 | f \mathbf{M}_2 | \mathbf{M}_3] \in \Gamma^{(p)}$. Thus $f \vartriangleright \Gamma^{(p)}$ for all $p \geq 0$.

Assume then that $f \vartriangleright \Gamma^{(p)}$ for all $p \geq 0$. Let $\mathbf{M} := [\mathbf{M}_1 | \mathbf{M}_2 | \mathbf{M}_3] \in \Gamma$, and let $q$ be the number of columns in $\mathbf{M}$. Then $\mathbf{M} \in \Gamma^{(q)}$ and hence $[\mathbf{M}_1 | f \mathbf{M}_2 | \mathbf{M}_3] \in \Gamma^{(q)}$ by our assumption. Since $\Gamma^{(q)} \subseteq \Gamma$, we have that $[\mathbf{M}_1 | f \mathbf{M}_2 | \mathbf{M}_3] \in \Gamma$, and we conclude that $f \vartriangleright \Gamma$.
\end{proof}

We say that a set $\mathcal{M} \subseteq \mathcal{M}_A$ of matrix collections is \emph{closed under quotients,} if for any $\Gamma \in \mathcal{M}$, every left and right quotient $\mathbf{N} \backslash \Gamma$ and $\Gamma / \mathbf{N}$ is also in $\mathcal{M}$. We say that $\mathcal{M}$ is \emph{closed under dividends,} if for every $\Gamma \in \mathcal{M}_A$, say of arity $m$, it holds that $\Gamma \in \mathcal{M}$ whenever there is an integer $p \geq 0$ such that $\Omega_m^{(p)} \subseteq \Gamma$ and $\mathbf{N}_1 \backslash \Gamma / \mathbf{N}_2 \in \mathcal{M}$ for all $m$-row matrices $[\mathbf{N}_1 | \mathbf{N}_2]$ with at least $p$ columns. We say that $\mathcal{M}$ is \emph{locally closed,} if $\Gamma \in \mathcal{M}$ whenever $\Gamma^{(p)} \in \mathcal{M}$ for all $p \geq 0$. We say that $\mathcal{M}$ is \emph{closed under unions,} if $\bigcup_{j \in J} \Gamma_j \in \mathcal{M}$ whenever $(\Gamma_j)_{j \in J}$ is a family of $m$-ary matrix collections from $\mathcal{M}$. We say that $\mathcal{M}$ is \emph{closed under formation of conjunctive minors,} if all conjunctive minors of nonempty families of members of $\mathcal{M}$ are members of $\mathcal{M}$.

\begin{theorem}
\label{thm:bundles}
Let $A$ be an arbitrary nonempty set. For any set $\mathcal{M}$ of matrix collections on $A$, the following two conditions are equivalent:
\begin{compactenum}[(i)]
\item $\mathcal{M}$ is locally closed and contains the binary equality matrix collection, the unary empty matrix collection, and all unary trivial matrix collections of breadth $p \geq 0$, and it is closed under formation of conjunctive minors, unions, quotients, and dividends.
\item $\mathcal{M}$ is characterized by some set of operations on $A$.
\end{compactenum}
\end{theorem}

We need to extend the notions of an $n$-tuple and a matrix and allow them to have infinite length or an infinite number of rows, as will be explained below. Operations remain finitary. These extended notions have no bearing on Theorem~\ref{thm:bundles} itself; they are only needed as a tool in its proof.

For any non-zero, possibly infinite ordinal $m$ (an ordinal $m$ is the set of lesser ordinals), an $m$-tuple of elements of $A$ is a map $m \to A$. Matrices can have infinitely many rows but only a finite number of columns: an $m \times n$ matrix $\mathbf{M}$, where $n$ is finite but $m$ may be finite or infinite, is an $n$-tuple of $m$-tuples $\mathbf{M} := (\vect{a}_{*1}, \dotsc, \vect{a}_{*n})$. The arities of matrix collections are allowed to be arbitrary non-zero, possibly infinite ordinals $m$ accordingly. In minor formations schemes, the target $m$ and the members $n_j$ of the source family are also allowed to be arbitrary non-zero, possibly infinite ordinals. We use the terms \emph{conjunctive $\infty$-minor} and \emph{simple $\infty$-minor} to refer to conjunctive minors and simple minors via a scheme whose target and source ordinals may be finite or infinite. The use of the term ``minor'' without the prefix ``$\infty$'' continues to mean the respective minor via a scheme whose target and source ordinals are all finite.

For a set $\mathcal{M}$ of matrix collections on $A$ of arbitrary, possibly infinite arities, we denote by $\mathcal{M}^\infty$ the set of those matrix collections which are conjunctive $\infty$-minors of families of members of $\mathcal{M}$. This set $\mathcal{M}^\infty$ is the smallest set of matrix collections containing $\mathcal{M}$ which is closed under formation of conjunctive $\infty$-minors, and it is called the \emph{conjunctive $\infty$-minor closure} of $\mathcal{M}$. Analogously to the proof of Lemma~4.7 and Corollary~4.8 in~\cite{Lclusters}, considering the formation of repeated conjunctive $\infty$-minors, we can show that the following holds:

\begin{corollary}
\label{cor:finitary}
Let $\mathcal{M}$ be a set of finitary matrix collections on $A$, and let $\mathcal{M}^\infty$ be its conjunctive $\infty$-minor closure. If $\mathcal{M}$ is closed under formation of conjunctive minors, then $\mathcal{M}$ is the set of all finitary matrix collections belonging to $\mathcal{M}^\infty$.
\end{corollary}

\begin{lemma}
\label{lemma:bundlesep}
Let $A$ be an arbitrary, possibly infinite nonempty set. Let $\mathcal{M}$ be a locally closed set of finitary matrix collections on $A$ that contains the binary equality matrix collection, the unary empty matrix collection, and all unary trivial matrix collections of breadth $p \geq 0$, and is closed under formation of conjunctive minors, unions, quotients, and dividends. Let $\mathcal{M}^{\infty}$ be the conjuctive $\infty$-minor closure of $\mathcal{M}$. Let $\Gamma \in \mathcal{M}_A \setminus \mathcal{M}$ be finitary. Then there exists an operation $g \in \mathcal{O}_A$ that preserves every member of $\mathcal{M}^\infty$ but does not preserve $\Gamma$.
\end{lemma}

\begin{proof}
Let $\Gamma$ be a finitary matrix collection on $A$ that is not in $\mathcal{M}$. Note that, by Corollary~\ref{cor:finitary}, $\Gamma$ cannot be in $\mathcal{M}^\infty$. Let $m$ be the arity of $\Gamma$. Since $\mathcal{M}$ is locally closed and $\Gamma$ does not belong to $\mathcal{M}$, there is an integer $p$ such that $\Gamma^{(p)} = \Gamma \cap \Omega_m^{(p)} \notin \mathcal{M}$; let $n$ be the smallest such integer. Every operation that does not preserve $\Gamma^{(n)}$ does not preserve $\Gamma$ either, so we can consider $\Gamma^{(n)}$ instead of $\Gamma$. Due to the minimality of $n$, the breadth of $\Gamma^{(n)}$ is $n$. Observe that $\Gamma$ is not the trivial matrix collection of breadth $n$ nor the empty matrix collection, because these are members of $\mathcal{M}$. Thus, $n \geq 1$.

We can assume that $\Gamma$ is a minimal nonmember of $\mathcal{M}$ with respect to identification of rows, i.e., every simple minor of $\Gamma$ corresponding to identification of some rows of $\Gamma$ is a member of $\mathcal{M}$. If this is not the case, then we can identify some rows of $\Gamma$ to obtain a minimal nonmember $\Gamma'$ of $\mathcal{M}$ and consider the matrix collection $\Gamma'$ instead of $\Gamma$. Note that by Lemma~\ref{lemma:conjmin}, every function not preserving $\Gamma'$ does not preserve $\Gamma$ either.

We can also assume that $\Gamma$ is a minimal nonmember of $\mathcal{M}$ with respect to quotients, i.e., whenever $\mathbf{N}$ is a nonempty $m$-row matrix, we have that $\mathbf{N} \backslash \Gamma \in \mathcal{M}$ and $\Gamma / \mathbf{N} \in \mathcal{M}$. If this is not the case, then consider a minimal nonmember $\mathbf{N} \backslash \Gamma$ or $\Gamma / \mathbf{N}$ of $\mathcal{M}$ in place of $\Gamma$. By Lemma~\ref{lemma:quotient}, every function not preserving $\mathbf{N} \backslash \Gamma$ or $\Gamma / \mathbf{N}$ does not preserve $\Gamma$ either.

The fact that $\Gamma$ is a minimal nonmember of $\mathcal{M}$ with respect to quotients implies that $\Omega_m^{(1)} \not\subseteq \Gamma$. For, suppose, on the contrary, that $\Omega_m^{(1)} \subseteq \Gamma$. Since all matrix collections $\mathbf{N}_1 \backslash \Gamma / \mathbf{N}_2$ such that $[\mathbf{N}_1 | \mathbf{N}_2]$ is a nonempty $m$-row matrix are in $\mathcal{M}$ and $\mathcal{M}$ is closed under dividends, we have that $\Gamma \in \mathcal{M}$, a contradiction.

Let $\Psi := \bigcup \{\Gamma' \in \mathcal{M} \mid \Gamma' \subseteq \Gamma\}$, i.e., $\Psi$ is the largest matrix collection in $\mathcal{M}$ such that $\Psi \subseteq \Gamma$. Note that this is not the empty union, because the empty matrix collection is a member of $\mathcal{M}$. Since $\Psi \in \mathcal{M}$ and $\Gamma \notin \mathcal{M}$, we obviously have that $\Psi \neq \Gamma$. Since $n$ was chosen to be the smallest integer satisfying $\Gamma^{(n)} \notin \mathcal{M}$, we have that $\Gamma^{(n-1)} \in \mathcal{M}$ and since $\Gamma^{(n-1)} \subseteq \Gamma^{(n)}$, it holds that $\Gamma^{(n-1)} \subseteq \Psi$. Thus there is an $m \times n$ matrix $\mathbf{D} := (\vect{d}_{*1}, \dotsc, \vect{d}_{*n})$ such that $\mathbf{D} \in \Gamma$ but $\mathbf{D} \notin \Psi$.

The rows of $\mathbf{D}$ are pairwise distinct. For, suppose, for the sake of contradiction, that rows $i$ and $j$ of $\mathbf{D}$ coincide. Since $\Gamma$ is a minimal nonmember of $\mathcal{M}$ with respect to identification of rows, by identifying rows $i$ and $j$ of $\Phi$ we obtain a matrix collection $\Gamma'$ that is in $\mathcal{M}$. By adding a dummy row in the place of the row that got deleted when we identified rows $i$ and $j$, and finally by intersecting with the conjunctive minor of the binary equality matrix collection whose $i$-th and $j$-th rows are equal (the overall effect of all these operations being the selection of those matrices in $\Gamma$ whose $i$-th and $j$-th rows coincide), we obtain a matrix collection in $\mathcal{M}$ that contains $\mathbf{D}$ and is a subset of $\Gamma$. But this is impossible by the choice of $\mathbf{D}$.

Let $\Upsilon := \bigcap \{\Gamma' \in \mathcal{M} \mid \mathbf{D} \in \Gamma'\}$, i.e., $\Upsilon$ is the smallest matrix collection in $\mathcal{M}$ that contains $\mathbf{D}$ as an element. Note that this is not the empty intersection, because the trivial matrix collection $\Omega_m^{(n)}$ is a member of $\mathcal{M}$ that contains $\mathbf{D}$. By the choice of $\mathbf{D}$, $\Upsilon \not\subseteq \Gamma$.

Consider the matrix collection $\hat{\Gamma} := \Gamma \cup \Omega_m^{(1)}$. We claim that if $[\mathbf{N}_1 | \mathbf{N}_2]$ is a nonempty $m$-row matrix, then $\mathbf{N}_1 \backslash \hat{\Gamma} / \mathbf{N}_2 = \mathbf{N}_1 \backslash \Gamma / \mathbf{N}_2$ or $\mathbf{N}_1 \backslash \hat{\Gamma} / \mathbf{N}_2 = \mathbf{N}_1 \backslash \Gamma / \mathbf{N}_2 \cup \{()\}$. By Lemma~\ref{lemma:quotprop}, $\mathbf{N}_1 \backslash \hat{\Gamma} / \mathbf{N}_2 = (\mathbf{N}_1 \backslash \Gamma / \mathbf{N}_2) \cup (\mathbf{N}_1 \backslash \Omega_m^{(1)} / \mathbf{N}_2)$. If $[\mathbf{N}_1 | \mathbf{N}_2]$ has more than one column, then $\mathbf{N}_1 \backslash \Omega_m^{(1)} / \mathbf{N}_2 = \emptyset$; if $[\mathbf{N}_1 | \mathbf{N}_2]$ has precisely one column, then $\mathbf{N}_1 \backslash \Omega_m^{(1)} / \mathbf{N}_2 = \{()\}$. The claim thus follows.

Since $\Gamma$ is a minimal nonmember of $\mathcal{M}$ with respect to quotients and since $\mathcal{M}$ is closed under unions and $\{()\} = \Omega_m^{(0)} \in \mathcal{M}$, by the above claim we have that $\mathbf{N}_1 \backslash \hat{\Gamma} / \mathbf{N}_2 \in \mathcal{M}$ whenever $[\mathbf{N}_1 | \mathbf{N}_2] \neq ()$. Since $\mathcal{M}$ is closed under dividends, we have that $\hat{\Gamma} \in \mathcal{M}$, and hence $\Upsilon \subseteq \hat{\Gamma}$. Thus, there exists a $m \times 1$ matrix $\vect{s} \in A^m$ such that $\vect{s} \in \Upsilon \setminus \Gamma$.

Let $M := (\vect{m}_{*1}, \dotsc, \vect{m}_{*n})$ be a $\mu \times n$ matrix whose first $m$ rows are the rows of $\mathbf{D}$ (i.e., $\vect{m}_{i*} = \vect{d}_{i*}$ for every $i \in m$) and whose other rows are the remaining distinct $n$-tuples in $A^n$; every $n$-tuple in $A^n$ is a row of $\mathbf{M}$ and there is no repetition of rows in $\mathbf{M}$. Note that $m \leq \mu$ and $\mu$ is infinite if and only if $A$ is infinite.

Let $\Theta := \bigcap \{\Gamma' \in \mathcal{M}^\infty \mid \mathbf{M} \in \Gamma'\}$. There must exist a $\mu$-tuple $\vect{u} := (u_t \mid t \in \mu)$ in $A^\mu$ such that $\vect{u}(i) = \vect{s}(i)$ for all $i \in m$ and $\vect{u} \in \Theta$. For, if this is not the case, then the projection of $\Theta$ to its first $m$ coordinates would be a member of $\mathcal{M}$ containing $\mathbf{D}$ but not containing $\vect{s}$, contradicting the choice of $\vect{s}$.

We can now define a function $g \colon A^n \to A$ by the rule $g \mathbf{M} = \vect{u}$. The definition is valid, because every $n$-tuple in $A^n$ occurs exactly once as a row of $\mathbf{M}$. It is clear that $g$ does not preserve $\Gamma$, because $\mathbf{D} \in \Gamma$ but $g \mathbf{D} = \vect{s} \notin \Gamma$.

We need to show that every matrix collection in $\mathcal{M}$ is preserved by $g$. Suppose, on the contrary, that there is a $\rho$-ary matrix collection $\Gamma_0 \in \mathcal{M}$ which is not preserved by $g$. Thus, for some $\rho \times r$ matrix $\mathbf{N} := [\mathbf{N}_1 | \mathbf{N}_2 | \mathbf{N}_3] \in \Gamma_0$ with $\mathbf{N}_2 = (\vect{c}_{*1}, \dotsc, \vect{c}_{*n})$, we have $[\mathbf{N}_1 | g \mathbf{N}_2 | \mathbf{N}_3] \notin \Gamma_0$. Let $\Gamma_1 := \mathbf{N}_1 \backslash \Gamma_0 / \mathbf{N}_3$. Since $\mathcal{M}$ is closed under quotients, $\Gamma_1 \in \mathcal{M}$. We have that $\mathbf{N}_1 \in \Gamma_1$ but $g \mathbf{N}_1 \notin \Gamma_1$, so $g$ does not preserve $\Gamma_1$ either. Define $h \colon \rho \to \mu$ to be any map such that
\[
\bigl( \vect{c}_{*1}(i), \dotsc, \vect{c}_{*n}(i) \bigr) = \bigl( (\vect{m}_{*1} h)(i), \dotsc, (\vect{m}_{*n} h)(i) \bigr)
\]
for every $i \in \rho$, i.e., row $i$ of $\mathbf{N}_2$ is the same as row $h(i)$ of $\mathbf{M}$, for each $i \in \rho$. Let $\Gamma_h$ be the $\mu$-ary simple $\infty$-minor of $\Gamma_1$ via $H := \{h\}$. Note that $\Gamma_h \in \mathcal{M}^\infty$.

We claim that $\mathbf{M} \in \Gamma_h$. To prove this, it is enough to show that $(\vect{m}_{*1} h, \dotsc, \linebreak[0] \vect{m}_{*n} h) \in \Gamma_1$. In fact, we have for $1 \leq j \leq n$,
\[
\vect{m}_{*j} h = (\vect{m}_{*j} h(i) \mid i \in \rho) = (\vect{c}_{*j}(i) \mid i \in \rho) = \vect{c}_{*j},
\]
and $(\vect{c}_{*1}, \dotsc, \vect{c}_{*n}) = \mathbf{N}_2 \in \Gamma_1$.

Next we claim that $\vect{u} \notin \Gamma_h$. For this, it is enough to show that $\vect{u}h \notin \Gamma_1$. For every $i \in \rho$, we have
\[
\begin{split}
(\vect{u} h)(i)
&= \bigl( g(\vect{m}_{*1}, \dotsc, \vect{m}_{*n}) h \bigr) (i) \\
&= g \bigl( (\vect{m}_{*1} h)(i), \dotsc, (\vect{m}_{*n} h)(i) \bigr)
= g \bigl( \vect{c}_{*1}(i), \dotsc, \vect{c}_{*n}(i) \bigr).
\end{split}
\]
Thus $\vect{u} h = g \mathbf{N}_1$. Since $g \mathbf{N}_1 \notin \Gamma_1$, we conclude that $\vect{u} \notin \Gamma_h$.

Thus $\Gamma_h$ is a matrix collection in $\mathcal{M}^\infty$ that contains $\mathbf{M}$ but does not contain $\vect{u}$. By the choice of $\vect{u}$, this is impossible. We conclude that $g$ preserves every matrix collection in $\mathcal{M}$.
\end{proof}

\begin{proof}[Proof of Theorem~\ref{thm:bundles}]
$\text{(ii)} \Rightarrow \text{(i)}$: It is clear that every operation on $A$ preserves the equality, empty, and trivial matrix collections. By Lemmas~\ref{lemma:conjmin}, \ref{lemma:union}, \ref{lemma:quotient}, and~\ref{lemma:dividend}, $\mathcal{M}$ is closed under formation of conjunctive minors, unions, quotients, and dividends. $\mathcal{M}$ is locally closed by Lemma~\ref{lemma:breadth}.

\smallskip
$\text{(i)} \Rightarrow \text{(ii)}$: Let $\mathcal{M}$ be a set of finitary matrix collections satisfying the conditions of (i). By Lemma~\ref{lemma:bundlesep}, for every matrix collection $\Gamma \in \mathcal{M}_A \setminus \mathcal{M}$, there is an operation $g \in \mathcal{O}_A$ that preserves all matrix collections in $\mathcal{M}$ but does not preserve $\Gamma$. Thus, the set of all these ``separating'' operations, for all $\Gamma \in \mathcal{M}_A \setminus \mathcal{M}$, constitutes a set characterizing $\mathcal{M}$.
\end{proof}

\section*{Acknowledgements}

The author would like to thank Jörg Koppitz for fruitful discussions on the topic.

\end{document}